\documentclass[a4paper,11pt]{article}
\usepackage{amssymb}
\usepackage{amsmath}
\usepackage{latexsym,a4wide}
\usepackage{hyperref}
\input xypic
\input xy
\xyoption{all}

\newtheorem{example}{Example}[section]
\newtheorem{definition}[example]{Definition}
\newtheorem{proposition}[example]{Proposition}

\newtheorem{theorem}[example]{Theorem}

\newtheorem{notation}[example]{Notation}

\newtheorem{lemma}[example]{Lemma}

\newtheorem{remark}[example]{Remark}
\newenvironment{proof}{\noindent \textbf{Proof:} }{$\Box$ \mbox{}}

\newcommand{\im}{\mathrm{Im}}

\begin{document}

\date{}
\author{Alper ODABA\c{S} and Erdal ULUALAN}
\title{Crossed Modules of Algebras as Ideal Maps}
\maketitle

\begin{abstract}
In this work, we  explore the close relationship between an ideal map structure $S\rightarrow End(R)$ on a homomorphism of commutative $k$-algebras $R\rightarrow S$  and an ideal simplicial algebra structure on the associated bar construction $Bar(S,R)$.
\end{abstract}

\section*{Introduction}

In this paper, we consider  the equivalence between the category of crossed modules of algebras (cf. \cite{porter1}) and the category of simplicial
commutative algebras with Moore complex of length 1 given in \cite{arvasi}. The main  aim of this note is to associate an explicit ideal simplicial algebra structure on the bar construction given a crossed module of algebras. We observed that a crossed module structure $(S\rightarrow End(R))$ or an ideal map structure  on a homomorphism of algebras $\eta:R\rightarrow S$, yields directly a simplicial algebra structure on the usual bar construction namely on the simplicial $\mathsf{k}$-module
$$
Bar(S,R)=(S\times R^{k})_{k\geqslant 0}.
$$
Thus $Bar(S,R)$ is isomorphic, as a simplicial $\mathsf{k}$-module, to a simplicial algebra which is compatible with the action of $R$  on the bar construction. Moreover this process is reversible.  Therefore we can summarize the result as follows: Given an algebra homomorphism $\eta:R\rightarrow S$, a crossed module structure or an ideal map structure on the homomorphism $\eta$ gives an ideal simplicial algebra structure  on the simplicial $\mathsf{k}$-module $Bar(S,R)$, and conversely, any ideal simplicial algebra structure on the simplicial $\mathsf{k}$-module $Bar(S,R)$ determines a crossed module  structure on the homomorphism $\eta$. These two explicit associations are mutual inverses. In the last section,  we are explaining how to give an extension of this result to Ellis's (crossed) squares of $k$-algebras (cf. \cite{ellis}). In section 5, considering  a \textit{crossed ideal structure} over the map $\alpha:\eta_1\rightarrow \eta_2$ between crossed modules $\eta_1$ and $\eta_2$, we proved  that a \textit{crossed ideal  map}  preserves the \textit{crossed ideals} in the category of crossed modules of commutative $k$-algebras.

These constructions in the category of  groups can be found in \cite{far1}. In fact, the results and general methods given in this work  are inspired by those proved for the corresponding case of groups using  homotopy normal maps in \cite{far1}. For further work about homotopy normal maps, see \cite{far2} and \cite{prez} and for the free normal closure of a homotopy normal map  see \cite{far3}.

\section{Simplicial sets and simplicial algebras}
Let $\mathsf{k}$ be a fixed commutative ring with identity. By a $\mathsf{k}$-algebra, we mean
a unitary $\mathsf{k}$-bimodule  $C$ endowed with a $\mathsf{k}$-bilinear associative multiplication $C\times C\rightarrow C$, $(c,c')\mapsto cc'$. The
algebra $C$ will as usual be called commutative if $cc'=c'c$ for all $c,c'\in C$. In this work, all algebras will be commutative and will be over the same fixed commutative ring $\mathsf{k}$. We will denote the category of all algebras over the commutative ring $\mathsf{k}$  by $\mathsf{Alg}$.

A simplicial  set $\mathsf{E}$ consists of a family
of sets $\{E_{n}\}$ together with face and degeneracy maps $d_{i}=d_{i}^{n}:E_{n}\rightarrow E_{n-1},$ \ \ $0\leqslant i\leqslant n$, \ $(n\neq 0)$ and $s_{i}=s_{i}^{n}:E_{n}\rightarrow E_{n+1}$, \ \
$0\leqslant i\leqslant n$, satisfying the usual simplicial
identities given in Andr\'{e} \cite{andre} or Illusie \cite{ill} for
example. It can be completely described as a functor
$\mathsf{E}:\Delta^{op}\rightarrow \mathsf{Sets}$ where $\Delta $ is
the category of finite ordinals $[n]=\{0<1<\cdots <n\}$
and non-decreasing maps.

We say that the simplicial set $\mathsf{E}$ is a simplicial $\mathsf{k}$-module (or $\mathsf{k}$-algebra) if $E_k$ is a $\mathsf{k}$-module (or a $\mathsf{k}$- algebra), for all $k$ and the face and degeneracy maps are homomorphisms of $\mathsf{k}$-modules (or  $\mathsf{k}$-algebras). Thus, a simplicial algebra can be defined as  a functor from the opposite category $\Delta ^{op}$ to $\mathsf{Alg}$.

\subsection{The simplicial $\mathsf{k}$-module $Bar(X,R)$}
In this section, we use  the usual bar construction of a simplicial $\mathsf{k}$-module by using the action of a $\mathsf{k}$-algebra on a $\mathsf{k}$-module.  First, we define  this action.

Let $R$ be a $\mathsf{k}$-algebra and $X$ be a $\mathsf{k}$-module. The action of $R$ on the $\mathsf{k}$-module $X$ is defined by the function $X\times R\rightarrow X$, $x:r\mapsto x^r$ (where $r\in R,x\in X$) satisfying the following conditions:
\begin{enumerate}
\item $(x)^{(r_1+r_2)}=(x^{r_1})^{r_2}$
\item $x^{0_{R}}=x$
\item $(x_1+x_2)^{r_1+r_2}=(x_1)^{r_1}+(x_2)^{r_2}$
\item $k(x)^r=(kx)^{kr}$
\end{enumerate}
for all $r,r_1,r_2\in R, x,x_1,x_2\in X, k\in \mathsf{k}$.

Let $R $ be a $\mathsf{k}$-algebra acting on the $\mathsf{k}$-module $X$. Then we obtain a $\mathsf{k}$-module $B_n=X\times R^{n}$
together with the operations
$$(x,r_1,r_2,\ldots,r_n)\oplus (x',r'_1,r'_2,\ldots,r'_n)=(x+x',r_1+r'_1,\ldots,r_n+r'_n)$$
for $x,x'\in X$ and $r_i,r'_i\in R_i$ and
$$k(x,r_1,r_2,\ldots,r_n)=(kx,kr_1,kr_2,\ldots,kr_n)$$
for $k\in \mathsf{k}$.

The bar construction
$$
B:=Bar(X,R)
$$
is the simplicial $\mathsf{k}$-module consisting of the following data.

\begin{enumerate}
\item for each integer $n\geqslant 0$, a $\mathsf{k}$-module  $B_n$ defined by $B_0=X$ for $n=0$, and $B_n=X\times R^{n}$ is the $\mathsf{k}$-module as described above for $n\geqslant 1$, together with

\item the face $\mathsf{k}$-module homomorphisms $d_i^n:d_i:B_n\rightarrow B_{n-1}$ for all $n\geqslant 1$ and $0\leqslant i\leqslant n$ defined by:

$(i)$ $d_0:(x,r_1,r_2,\ldots, r_n)\mapsto (x^{r_1},r_2,\ldots, r_n)$

$(ii)$ $d_i:(x,r_1,r_2,\ldots,r_i,r_{i+1},\ldots, r_n)\mapsto(x,r_1,r_2,\ldots,r_i+r_{i+1},\ldots, r_n)$ for $1\leqslant i <n$,

$(iii)$ $d_n:(x,r_1,r_2,\ldots, r_n)\mapsto(x,r_1,r_2,\ldots, r_{n-1}),$

\item and together with degeneracy $\mathsf{k}$-module homomorphisms; $s_i:B_n \rightarrow B_{n+1}$ defined by
$$
s_i:(x,r_1,r_2,\ldots, r_n)\mapsto(x,r_1,r_2,\ldots,r_i,0,r_{i+1},\ldots, r_n)
$$
for all $n\geqslant 0$ and $0\leqslant i \leqslant n$.
\end{enumerate}
In this construction we show briefly that $d_0$ is a $\mathsf{k}$-module homomorphism from $B_n$ to $B_{n-1}$. For
$u=(x,r_1,r_2,\ldots, r_n)$,$v=(x',r'_1,r'_2,\ldots, r'_n)\in B_n$  and $k\in \mathsf{k}$, we obtain
\begin{align*}
d_0(u\oplus v)&=d_0(x+x',r_1+r'_1,\ldots,r_n+r'_n)\\
&=((x+x')^{r_1+r'_1},r_2+r'_2,\ldots,r_n+r'_n)\\
&=(x^{r_1}+(x')^{r'_1},r_2+r'_2,\ldots,r_n+r'_n)\\
&=(x^{r_1},r_2,\ldots, r_n)\oplus((x')^{r'_1},r'_2,\ldots, r'_n)\\
&=d_0(u)\oplus d_0(v)
\intertext{ and }
d_0(ku)&=d_0(kx,kr_1,\ldots,kr_n)\\
&=((kx)^{kr_1},kr_2,\ldots,kr_n)\\
&=(k(x^{r_1}),kr_2,\ldots,kr_n)\\
&=k(x^{r_1},r_2,\ldots, r_n)\\
&=kd_0(u).
\end{align*}

\subsection{An ideal simplicial algebra structure on $Bar(S,R)$}

Suppose  now that $X=S$ is a  $\mathsf{k}$-algebra and the $\mathsf{k}$-algebra $R$ acts on the underlying $\mathsf{k}$-module $S$ of  the $\mathsf{k}$-algebra  $S$ via a homomorphism $\eta:R\rightarrow S$, i.e. the action is
$$
r:s\rightarrow s^ r=s +\eta(r)
$$
for all $r\in R$ and $s\in S$. In this case we obtain
\begin{enumerate}
\item $s^{(r_1+r_2)}=s+\eta(r_1+r_2)=(s+\eta(r_1))+\eta(r_2)=(s^{r_1})^{r_2}$
\item $s^{0_R}=s+\eta(0_R)=s+0_S=s$
\item $(s_1+s_2)^{(r_1+r_2)}=(s_1+s_2)+\eta(r_1+r_2)=s_1+\eta(r_1)+s_2+\eta(r_2)=(s_1)^{r_1}+(s_2)^{r_2}$
\item $k(s)^r=k(s+\eta(r))=ks+\eta(kr)=(ks)^{kr}$
\end{enumerate}
for all $s,s_1,s_2\in S$ and $r,r_1,r_2\in R$ and $k\in \mathsf{k}$. We denote the resulting simplicial  $\mathsf{k}$-module  by $Bar(S,R)$ suppressing the map $\eta$.

\begin{definition}
Let $B:=Bar(S,R)$. By an ideal simplicial algebra structure on $B$ we mean the following

$(i)$ $B_0=S$ is the $\mathsf{k}$-algebra $S$,

$(ii)$ $B_k:=S\times R^k$ for $k\geqslant 1$, is endowed with a $\mathsf{k}$-algebra structure for all $k\geqslant 1$, we denote the multiplication by
$$
(s,r_1,\ldots, r_k)\ast (s',r'_1,\ldots, r'_k).
$$

$(iii)$ the face  $d_i^k$ and the degeneracy  $s_j^k$ $\mathsf{k}$-module homomorphisms are $\mathsf{k}$-algebra homomorphisms.

$(iv)$ $$(s,0,\ldots,0)\ast(s',r'_1,\ldots, r'_k)=(ss',0,\ldots, 0)$$
for all $s,s'\in S$ and $(s',r'_1,\ldots, r'_k)\in B_k$ where the operations takes place in $B_k$.
\end{definition}

\begin{remark} \rm{}
By the natural action of $S$ on $Bar(S,R)$, we mean
$$
s:(s',r_1,\ldots,r_k)\mapsto s\cdot(s',r_1,\ldots,r_k)=(ss',r_1,\ldots,r_k)
$$
for all $k\geqslant 0$, $(s',r_1,\ldots,r_k)\in B_k$ and $s\in S$. When  we say that the multiplication in $Bar(S,R)$
is compatible with the natural action of $S$, we mean that condition $(iv)$ of the above definition holds.
\end{remark}
\begin{notation}\rm{}
 Let $k\geqslant 1$. We denote
\begin{enumerate}
\item $S_k:=\{(s,0_R,0_R,\ldots,0_R):s\in S\}$ is a subalgebra of $B_k$.
\item $R_k:=\{(0_S,r_1,r_2,\ldots,r_k):r_i\in R\}$ is an algebra ideal of $B_k$.
\end{enumerate}
\end{notation}
\begin{lemma}
Suppose that $Bar(S,R)$ is endowed with an ideal simplicial algebra structure. Let $ k\geqslant 1$. Then $S_k$ is an ideal of $B_k$ which is isomorphic to $S$, $R_k$ is an ideal of $B_k$, $B_k=S_k + R_k$ and $S_k\cap R_k=\{0\}$.
\end{lemma}
\begin{proof}
$S_k$ is the image of $S_{k-1}$ under $s_{k-1}$, so by induction it is a subalgebra of $B_k$ and $s_{k-1}$ is injective, it is isomorphic to $S$. Also, $R_k$ is the kernel of $d_k\circ d_{k-1}\circ \cdots \circ d_1$, so it is an ideal of $B_k$. Clearly $S_k\cap R_k=\{0\}$ and $B_k=S_k + R_k$.
\end{proof}

\subsection{Crossed modules, ideal maps and ideal structures}

Crossed modules of groups were initially defined by Whitehead in \cite{w}. The algebra analogue has been studied by Porter in \cite{porter1}.

A crossed module of algebras consists of an algebra homomorphism $\eta:R\rightarrow S$ which we call here an ideal map (see Remark \ref{2.5}) together with a  homomorphism $l:S\rightarrow End(R)$ which we call here an ideal structure (or a crossed module structure) on $\eta$ such that when denoting by $s\cdot r$ the image of $r\in R$ under $l_s$ for $s\in S$  which is satisfying the conditions below (for all $k\in \mathsf{k}$, $r,r'\in R$ and $s,s'\in S$)
\begin{enumerate}
\item $k(s\cdot r)=(ks)\cdot r= s\cdot (kr)$
\item  $s\cdot (r+r')=s\cdot r +s\cdot r'$
\item $(s+s')\cdot r=s\cdot r+s'\cdot r$
\item $s\cdot (rr')=(s\cdot r)r'=r(s\cdot r')$
\item $(ss')\cdot r=s\cdot(s'\cdot r) $
\end{enumerate}
and the following two requirements are satisfied:

$(CM1)$ $\eta(l_s(r))=s\eta(r)$ for all $s\in S$ and $r\in R$.

$(CM2)$ $l_{\eta(r)}(r')=rr'$ for all $r,r'\in R$.

\begin{remark} \rm{}\label{2.5}
Let $S$ and $R$ be algebras and let $\eta:R\rightarrow S$ be an algebra homomorphism.
If $l_s:S\rightarrow End(R); s\in S$ is a crossed module structure on the  homomorphism $\eta:R\rightarrow S$, then $\im(\eta)$ becomes an ideal of $S$. Indeed, for all $s\in S$ and $s'\in \im(\eta)$ with $s'=\eta(r); r\in R$, we obtain from $(CM1)$,
$$
ss'=s\eta(r)=\eta(l_s(r))\in \im(\eta).
$$
Thus $\im(\eta)$ is an ideal of $S$. Conversely, if $I$ is an ideal of the algebra $S$, then the inclusion map $I\rightarrow S$ is a crossed module with the natural action of $S$ on $I$. Further $\ker\eta$ is an ideal in $R$  and a module over $S$. The ideal $\im(\eta)$ of $S$ acts trivially on $\ker\eta$, hence $\ker\eta$ inherits an action of $S/\im(\eta)$ to become an  $S/\im(\eta)$-module.

Now let $S$ be an algebra and $R$ be subalgebra of $S$. Let $\eta:R\rightarrow S$ be the inclusion map and let $S/R$ be the set of cosets of $R$ in $S$. Then there is a natural action of $S$ on the set $S/R$ via left multiplication and it is easy to check that the following statements are equivalent.

$(i)$ $R$ is an ideal of $S$.

$(ii)$ There exists a crossed module structure on the inclusion map $\eta:R\rightarrow S$.

$(iii)$ There exists an algebra structure  on $S/R$ with the action of $S$ on $S/R$ given by
$$
s\cdot(s'+R)=ss'+R
$$
for all $s,s'\in S$.
\end{remark}

\section{From an ideal simplicial algebra structure on $Bar(S,R)$ to an ideal structure on the map $\eta:R\rightarrow S$}
In this section $R$ and $S$ are algebras and $\eta:R\rightarrow S$ is an algebra homomorphism. The purpose of this section is to prove that we can recover the crossed module structure (or an ideal structure) on a homomorphism between $\mathsf{k}$-algebras, from an ideal simplicial algebra structure on the associated bar construction.
\begin{lemma}\label{2.1}
Suppose that $Bar(S,R)$ is endowed with an ideal simplicial algebra structure. Then
\begin{enumerate}
\item $(0,r)\oplus(0,r')=(0,r+r')$ and $(0,r)\ast(0,r')=(0,rr')$
for all $r,r'\in R$ where the operations take place in $R_1$.
\item The map $l:S\rightarrow End(R)$ defined by
$$
l_s:(r)\mapsto s\cdot r
$$
gives an ideal structure (or a crossed module structure) on $\eta$,
where
$$
(0,s\cdot r)=(s,0)\ast (0,r).
$$
\end{enumerate}
\end{lemma}

\begin{lemma}
Let $k\geqslant 0$ and $r,r'\in R$. Then

$(i)$ The zero element of $B_k$ is $(0_S,0_R,\ldots,0_R)$,

$(ii)$ $(0,-r,r)\oplus (0,0,r')=(0,-r,r+r'),$

$(iii)$ $(-\eta(r),r)\oplus(0,r')=(-\eta(r),r+r'),$

$(iv)$ $(0,r)\ast(0,r')=(0,rr')$.
\end{lemma}

\begin{proof}

$(i)$ By definition, the zero element of $B_0=S$ is the zero element $0_S$ of $S$. Then by induction since $s_0:B_k\rightarrow B_{k+1}$ is an algebra homomorphism, for all $k\geqslant 0$, part $(i)$ follows. For example $s_0(0_S)=(0_S,0_R)$ is the zero element of $B_1$ and $s_0(0_S,0_R)=(0_S,0_R,0_R)$ is the zero element of $B_2$ and so on.

$(ii)$ Applying $d_2^2$ and using $(i)$ we get that
$$
(0_S,-r,r)\oplus(0_S,0_R,r')=(0_S,-r,x).
$$
Applying $d_1^2$ and using $(i)$ again we get that
$$
(0_S,r')=(0_S,-r+x)
$$
so; $x=r+r'$ and $(ii)$ holds.

$(iii)$ This part follows from $(ii)$ by applying $d_0^2$.
\end{proof}

Notice that $B_1=S\ltimes R$ is a semidirect product algebra of $R$ by $S$, that is, the addition and the multiplication in $B_1$ are given respectively by
$$
(s,r)\oplus(s',r')=(s+s',r+r')
$$
and
$$
(s,r)\ast(s',r')=(ss',s\cdot r' + s'\cdot r+rr')
$$
for all $s,s'\in S$ and $r,r'\in R$.

\begin{lemma}\label{3.3}
The map $\Phi:S\ltimes R \rightarrow S$ defined by $\Phi(s,r)=s+\eta(r)$ is a homomorphism of algebras if and only if $\eta$ satisfies $(CM1)$ above.
\end{lemma}
\begin{proof}
For all $(s,r),(s',r')\in S\ltimes R$, we have
\begin{align*}
\Phi((s,r)\oplus(s',r'))=&\Phi((s+s',r+r'))\\
=&s+s'+\eta(r+r')\\
=&s+\eta(r)+s'+\eta(r')\\
=&\Phi(s,r)+\Phi(s',r')
\intertext{ and }
\Phi((s,r)\ast(s',r'))=&\Phi(ss',s\cdot r' + s'\cdot r+rr')\\
=&ss'+\eta(s\cdot r' + s'\cdot r+rr')\\
=&ss'+s\eta(r')+s'\eta(r)+\eta(r)\eta(r')\ \ \text{ since } (CM1)\\
=&s(s'+\eta(r'))+\eta(r)(s'+\eta(r'))\\
=&(s+\eta(r))(s'+\eta(r'))\\
=&\Phi((s,r))\Phi((s',r')).
\end{align*}
\end{proof}
\begin{lemma}\label{3.4}
Consider the action of $R$ on itself via multiplication and form the semidirect product $R\ltimes R$ with respect to this action. Thus
$$
(a,b)\oplus(c,d)=(a+c,b+d)
$$
and
$$
(a,b)\ast(c,d)=(ac,ad+bc+bd), a,b,c,d\in R.
$$
Then the map $\Phi:R\ltimes R \rightarrow S\ltimes R$ defined by $(a,b)\mapsto (\eta(a),b)$ is a homomorphism if and only if $\eta$ satisfies $(CM2)$.
\end{lemma}
\begin{proof}
For all $(a,b),(c,d)\in R\ltimes R$, we obtain
\begin{align*}
\Phi((a,b)\oplus(c,d))=&\Phi((a+c,b+d))\\
=&(\eta(a+c),b+d)\\
=&(\eta(a),b)+(\eta(c),d)\\
=&\Phi(a,b)+\Phi(c,d)
\intertext{ and }
\Phi((a,b))\Phi((c,d))=&(\eta(a),b)(\eta(c),d)\\
=&(\eta a\eta c,\eta(a)\cdot d+\eta(c)\cdot b+bd )\\
=&(\eta(ac),ad+bc+bd) \ \text{ since } (CM2)\\
=&\Phi(ac,ad+bc+bd)\\
=&\Phi((a,b)\ast(c,d)).
\end{align*}
\end{proof}
\begin{lemma}\label{3.5}
Let $a_i, b_i \in R$.  Then

$(i)$
$$(0_S,a_1,\ldots,a_k)\ast(0_S,b_1,\ldots,b_k)=(0_S,a_1b_1,a_1b_2+a_2(b_1+b_2),\ldots, (\sum\limits_{i=1}^{k-1}a_i)b_k+a_k\sum\limits_{i=1}^{k}b_i).$$

$(ii)$ Let $s\in S$ and $(0_S,a_1,a_2,\ldots,a_k)\in R_k$. Then
$$
(0_S,a_1,\ldots,a_k)\ast(s,0_R,\ldots, 0_R)=(0_S,s\cdot a_1, s\cdot a_2,\ldots, s\cdot a_k).
$$
\end{lemma}
\begin{proof}
We prove $(i)$ by induction on $k$. For $k=1$, it is easy to see that
$$
(0_s,a_1)\ast(0_S,b_1)=(0_S,a_1b_1).
$$
Then by applying $d_k$ using induction we see that
$$
(0_S,a_1,\ldots,a_k)\ast(0_S,b_1,\ldots,b_k)=(0_S,a_1b_1,\ldots,(\sum\limits_{i=1}^{k-2}a_i)b_{k-1}+a_{k-1}\sum\limits_{i=1}^{k-1}b_i,x).
$$
Applying $d_{k-1}$ using induction once more we get that
\begin{multline*}
(0_S,a_1b_1,\ldots,(\sum\limits_{i=1}^{k-2}a_i)b_{k-1}+a_{k-1}\sum\limits_{i=1}^{k-1}b_i+x)\\
\begin{aligned}
=&(0_S,a_1,\ldots,a_{k-1}+a_k)\ast(0_S,b_1,\ldots,b_{k-1}+b_k)\\
=&(0_S, a_1b_1,\ldots,\sum\limits_{i=1}^{k-2}a_i(b_{k-1}+b_k)+(a_{k-1}+a_k)\sum\limits_{i=1}^{k}b_i)\\
=&(0_S, a_1b_1,\ldots,\sum\limits_{i=1}^{k-2}a_i(b_{k-1})+\sum\limits_{i=1}^{k-2}a_i(b_k)+a_{k-1}\sum\limits_{i=1}^{k}b_i+a_k\sum\limits_{i=1}^{k}b_i)\\
=&(0_S,a_1b_1,\ldots,\sum\limits_{i=1}^{k-2}a_i(b_{k-1})+a_{k-1}\sum\limits_{i=1}^{k-1}b_i+a_{k-1} b_k+\sum\limits_{i=1}^{k-2}a_i(b_k)+(a_k)\sum\limits_{i=1}^{k}b_i).\\
\end{aligned}
\end{multline*}
It follows that
$$
x=(\sum\limits_{i=1}^{k-1}a_i)b_k+a_k\sum\limits_{i=1}^{k}b_i.
$$

$(ii)$ By induction on $k$ similarly, we prove Part $(ii)$. For $k=1$, we have
$$
(0_S,a_1)\ast(s,0_R)=(0_S,s\cdot a_1).
$$
Applying $d_k$ using induction we see that for $k-1$
$$
(0_S,a_1,\ldots,a_k)\ast (s,0_R,\ldots, 0_R)=(0_S,s\cdot a_1, s\cdot a_2,\ldots, s\cdot a_{k-1},x).
$$
Then applying $d_{k-1}$ using induction, we get that
\begin{align*}
(0_S,s\cdot a_1,\ldots,s\cdot a_{k-1}+x)=&(0_S,a_1,\ldots,a_{k-1}+a_k)\ast(s,0_R,\ldots,0_R)\\
=&(0_S,s\cdot a_1,\ldots s\cdot (a_{k-1}+a_k))
\end{align*}
and so, $x=s\cdot a_k$.
\end{proof}

\begin{proposition}\label{3.6}
The homomorphism $l:S\rightarrow End(R)$ is an ideal structure (or a crossed module structure) on the map $\eta:R\rightarrow S$.
\end{proposition}
\begin{proof}
Since $B_1=S\ltimes R$, and since the homomorphism
$$d_0:S\ltimes R=B_1 \rightarrow B_0=S$$
is defined by $d_0(s,r)=s^r=s+\eta(r)$, Lemma \ref{3.3} implies that $(CM1)$ holds for the map $\eta:R\rightarrow S$. Notice that by Lemma \ref{3.5} the subalgebra $R_2$ is isomorphic to $R\ltimes R$. Further, the map $d_0$ restricted to $R_2$ is given by $d_0(0_S,a,b)=(\eta(a),b)$ and it is a homomorphism from $R\ltimes R$ to $S\ltimes R$ given by $(a,b)\mapsto (\eta(a),b)$. Hence by Lemma \ref{3.4}, $(CM2)$ holds for the map $\eta$.
\end{proof}

Let $(s,a_1,\ldots, a_k), (s',b_1,\ldots, b_k)\in B_k$. Then from the above results we get
\begin{align*}
(s,a_1,\ldots, a_k)\ast(s',b_1,\ldots, b_k)=&(ss',s\cdot b_1 +s'\cdot a_1+a_1b_1,s\cdot b_2+s'\cdot a_2+a_1b_2+a_2(b_1+b_2),\\
&\ldots, s\cdot b_k+s'\cdot a_k+\sum\limits_{i=1}^{k-1}a_ib_k+a_k\sum\limits_{i=1}^{k}b_i)
\end{align*}
and
$$
(s,a_1,\ldots, a_k)\oplus(s',b_1,\ldots, b_k)=(s+s',a_1+b_1,\ldots,a_k+b_k).
$$

\section{From an ideal structure on $\eta:R\rightarrow S$ to an ideal simplicial algebra structure on $Bar(S,R)$.}
In this section $S$ and $R$ are algebras and $\eta:R\rightarrow S$, $l:S\rightarrow End(R)$ are algebra homomorphism. Recall that we denote
$$
l_s:r\mapsto l_s(r)=s\cdot r
$$
for $s\in S$ and $r\in R$.

We assume that $l$ is an ideal structure or a crossed module structure on $\eta$. We let $Bar(S,R)$ denote the bar construction using the action of the $\mathsf{k}$-algebra $R$ on the underlying $\mathsf{k}$-module  $S$ of the algebra $S$ via $s\mapsto s+\eta(r)$ for all $s\in S$ and $r\in R$. Our aim in this section is to show that the crossed module structure $l$ leads to an ideal simplicial algebra structure on $Bar(S,R)$.

We start by defining a multiplication on $B_k$ for all $k\geqslant 0$. For $k=0$, $B_0=S$ and the operations are as in $S$. Obviously, from simplicial structure $Bar(S,R)$, for $k\geqslant 1$, we can denote the addition by
$$
(s,a_1,\ldots, a_k)\oplus(s',b_1,\ldots, b_k)=(s+s',a_1+b_1,\ldots, a_k+b_k).
$$
We can define the multiplication by
\begin{align*}
(s,a_1,\ldots, a_k)\ast(s',b_1,\ldots, b_k)=&(ss',s\cdot b_1 +s'\cdot a_1+a_1b_1,s\cdot b_2+s'\cdot a_2+a_1b_2+a_2(b_1+b_2),\\
&\ldots, s\cdot b_k+s'\cdot a_k+\sum\limits_{i=1}^{k-1}a_ib_k+a_k\sum\limits_{i=1}^{k}b_i)
\end{align*}
as illustrated above.

\begin{theorem}
Let $k\geqslant 0$. Then

$(i)$ $B_k$ is an algebra,

$(ii)$ the $\mathsf{k}$-module homomorphism
$$d_0:(s,a_1,\ldots, a_k)\mapsto (s+\eta(a_1),a_2,\ldots, a_k)$$
is a $\mathsf{k}$-algebra homomorphism from $B_k$ to $B_{k-1}$,

$(iii)$ the $\mathsf{k}$-module homomorphisms
$$
d_i:(s,a_1,\ldots, a_k)\mapsto (s,a_1,\ldots,a_{i-1}+a_i,\ldots, a_k)
$$
are $\mathsf{k}$-algebra homomorphisms from $B_k$ to $B_{k-1}$ for all $1\leqslant i\leqslant k-1$,

$(iv)$ the $\mathsf{k}$-module homomorphism
$$
d_k:(s,a_1,\ldots, a_k)\mapsto (s,a_1,\ldots,a_{k-1})
$$
is a $\mathsf{k}$-algebra homomorphism  from $B_k$ to $B_{k-1}$,

$(v)$ the $\mathsf{k}$-module homomorphisms
$$
s_i:(s,a_1,\ldots, a_k)\mapsto (s,a_1,\ldots,a_i,0,a_{i+1},\ldots, a_k)
$$
are $\mathsf{k}$-algebra homomorphisms for all $0\leqslant i\leqslant k$.
\end{theorem}
\begin{proof}
$(i)$ For each $k\geqslant 1$ define
$$
\eta_k:(s,a_1,\ldots, a_k)\mapsto s+\eta(a_1+\ldots+a_k)
$$
from $B_k$ to $S$. We prove that $B_k$ is an algebra and that $\eta_k$ is an algebra homomorphism. For $k=1$, this is Lemma \ref{3.3}. Suppose this holds for $k-1$. Then $B_{k-1}$ acts on $R$ via
$$
(s,a_1,\ldots, a_{k-1}):a\mapsto a\cdot (s+\eta(a_1+\ldots+a_{k-1}))
$$
for $(s,a_1,\ldots, a_{k-1})\in B_{k-1}$ and $a\in R$. Notice that $B_k$ is just the semi-direct product algebra $B_{k-1}\ltimes R$ with respect to this action, so $B_k$ is an algebra. To show that $\eta_k$ is an algebra homomorphism we obtain
\begin{multline*}
\eta_k((s,a_1,\ldots, a_k)\ast(s',b_1,\ldots, b_k))\\
\begin{aligned}
=&\eta_k(ss',s\cdot b_1 +s'\cdot a_1+a_1b_1,s\cdot b_2+s'\cdot a_2+a_1b_2+a_2(b_1+b_2),\\
&\ \ \ldots, s\cdot b_k+s'\cdot a_k+\sum\limits_{i=1}^{k-1}a_i b_k+a_k\sum\limits_{i=1}^{k}b_i)\\
=&ss'+\eta(s\cdot b_1 +s'\cdot a_1+a_1b_1+s\cdot b_2+s'\cdot a_2+a_1b_2+a_2(b_1+b_2)+\\
&\ \ \ldots+ s\cdot b_k+s'\cdot a_k+\sum\limits_{i=1}^{k-1}a_i b_k+a_k\sum\limits_{i=1}^{k}b_i)\\
=&ss'+s\eta(b_1)+s'\eta(a_1)+\eta(a_1)\eta(b_1)+s\eta(b_2)+s'\eta(a_2)+\eta(a_1b_2+a_2(b_1+b_2))\\
&\ \ \ldots+ s\eta(b_k)+s'\eta(a_k)+\sum\limits_{i=1}^{k-1}\eta(a_i)\eta(b_k)+\eta(a_k)\sum\limits_{i=1}^{k}\eta(b_i)\\
=&s(s'+\sum\limits_{i=1}^{k}\eta(b_i))+(\sum\limits_{i=1}^{k}\eta(a_i))(s'+\sum\limits_{i=1}^{k}\eta(b_i))\\
=&(s+\eta(a_1+\ldots+a_k))(s'+\eta(b_1+\ldots+b_k))\\
=&\eta_k(s,a_1,\ldots, a_k)\eta_k(s',b_1,\ldots, b_k).
\end{aligned}
\end{multline*}

$(ii)$ Let
$$
u=(s,a_1,\ldots,a_k),v=(s',b_1,\ldots,b_k)\in B_k.
$$
Then we obtain
\begin{align*}
d_0(u\ast v)=&d_0(ss',s\cdot b_1 +s'\cdot a_1+a_1b_1,s\cdot b_2+s'\cdot a_2+a_1b_2+a_2(b_1+b_2),\\
&\ \ \ldots, s\cdot b_k+s'\cdot a_k+\sum\limits_{i=1}^{k-1}a_i b_k+a_k\sum\limits_{i=1}^{k}b_i)\\
=&(ss'+s\eta(b_1)+s'\eta(a_1)+\eta(a_1)\eta(b_1),s\cdot b_2+s'\cdot a_2+a_1b_2+a_2(b_1+b_2),\\
&\ \ \ldots, s\cdot b_k+s'\cdot a_k+\sum\limits_{i=1}^{k-1}a_i b_k+a_k\sum\limits_{i=1}^{k}b_i)\\
=&((s+\eta(a_1)(s'+\eta(b_1)),s\cdot b_2+s'\cdot a_2+a_1b_2+a_2(b_1+b_2),\\
&\ \ \ldots, s\cdot b_k+s'\cdot a_k+\sum\limits_{i=1}^{k-1}a_i b_k+a_k\sum\limits_{i=1}^{k}b_i)\\
=&(s+\eta a_1,a_2,\ldots,a_k)(s'+\eta b_1,b_2,\ldots,b_k)\\
=&d_0(u)\ast d_0(v).
\end{align*}

$(iii)$ Let
$$
u=(s,a_1,\ldots,a_k),v=(s',b_1,\ldots,b_k)\in B_k.
$$
We shall show that the $\mathsf{k}$-module homomorphisms $d_i$ are $\mathsf{k}$-algebra homomorphisms for $0\leqslant i\leqslant k-1$. We calculate
\begin{align*}
d_i(u\ast v)=&d_i(ss',s\cdot b_1 +s'\cdot a_1+a_1b_1,s\cdot b_2+s'\cdot a_2+a_1b_2+a_2(b_1+b_2),\\
&\ \ \ldots, s\cdot b_k+s'\cdot a_k+\sum\limits_{i=1}^{k-1}a_i b_k+a_k\sum\limits_{i=1}^{k}b_i)\\
=&(ss',s\cdot b_1 +s'\cdot a_1+a_1b_1,s\cdot b_2+s'\cdot a_2+a_1b_2+a_2(b_1+b_2),\\
&\ldots, s\cdot b_{i-1}+s'\cdot a_{i-1}+b_{i-1}\sum\limits_{j=1}^{i-2}a_j +a_{i-1}\sum\limits_{j=1}^{i-1}b_j\\
&+s\cdot b_{i}+s'\cdot a_{i}+b_{i}\sum\limits_{j=1}^{i-1}a_j +a_{i}\sum\limits_{j=1}^{i}b_j,\\
&\ \ \ldots, s\cdot b_k+s'\cdot a_k+\sum\limits_{i=1}^{k-1}a_i b_k+a_k\sum\limits_{i=1}^{k}b_i)\\
=&(ss',s\cdot b_1 +s'\cdot a_1+a_1b_1,\ldots,s'\cdot(a_{i-1}+a_i)+s\cdot(b_{i-1}+b_i)\\
&\ \ +(b_{i-1}+b_i)\sum\limits_{j=1}^{i-2}a_j+(a_{i-1}+a_i)\sum\limits_{j=1}^{i-1}b_j+(a_{i-1}+a_i)b_i,\\
&\qquad \ldots, s\cdot b_k+s'\cdot a_k+\sum\limits_{i=1}^{k-1}a_i b_k+a_k\sum\limits_{i=1}^{k}b_i)\\
=&(s,a_1,\ldots,a_{i-1}+a_i,\ldots, a_k)(s',b_1,\ldots,b_{i-1}+b_i,\ldots, b_k)\\
=&d_i(u)\ast d_i(v)
\end{align*}
for $0\leqslant i\leqslant k-1$, so Part $(iii)$ holds.

$(iv)$ Since in any semi-direct product, projection onto the first coordinate is a homomorphism, the projection map
$$
d_k:(s,a_1,\ldots, a_k)\mapsto (s,a_1,\ldots,a_{k-1})
$$
is a homomorphism  from $B_k$ to $B_{k-1}$ for $k\geqslant 1$.

$(v)$ We leave it to the reader.
\end{proof}
\section{The mutual inverse  relation between  above associations}
Let $\eta:R\rightarrow S$ be an algebra homomorphism. We showed that how to start with an ideal simplicial algebra structure on $Bar(S,R)$ and obtain a crossed module structure $l:S\rightarrow End(R)$ on $\eta$ and we showed how to start with a crossed module structure on $\eta$ and obtain an ideal simplicial algebra structure on the simplicial $\mathsf{k}$-module $Bar(S,R)$. Our aim in this section is to make the observation that these two associations are mutual inverses.

Assume first that the simplicial $\mathsf{k}$-module $Bar(S,R)$ is  endowed with an ideal simplicial algebra structure, and denote the multiplication in $B_k$ as
$$
(s,a_1,\ldots, a_k)\ast (s',b_1,\ldots, b_k).
$$
We showed that the action $l:S\rightarrow End(R)$ given by $l_s:r\mapsto s\cdot r$  gives an crossed module structure on $\eta$. Further given this crossed module structure on $\eta$ the equation
\begin{align*}
(s,a_1,\ldots, a_k)\ast(s',b_1,\ldots, b_k)=&(ss',s\cdot b_1 +s'\cdot a_1+a_1b_1,s\cdot b_2+s'\cdot a_2+a_1b_2+a_2(b_1+b_2),\\
&\ldots, s\cdot b_k+s'\cdot a_k+\sum\limits_{i=1}^{k-1}a_ib_k+a_k\sum\limits_{i=1}^{k}b_i).
\end{align*}
given above tells us how to define an ideal simplicial algebra structure on $B_k$ with the multiplication `$\ast$'.

Conversely let $l:S\rightarrow End(R)$ be an ideal structure (or a crossed module structure) on $\eta$. Let  `$\ast$' be the multiplication  in $B_k$ as given above. Let $l':S\rightarrow End(R)$ be the crossed module structure on $\eta$.  That is for all $s\in S$, $l'_s:r\mapsto s'$ where $(0,s')=(s,0)\ast (0,r)$. Now by definition of the operation $\ast$, we obtain
$$
(s,0)\ast(0,r)=(0s,s\cdot r+0.0+0\cdot r)=(0,s\cdot r).
$$
We thus see that $s'=s\cdot r$ for all $r\in R$, $s\in S$, that is $l'_s=l_s$ for all $s\in S$. This completes the observation that the two associations are mutual inverses.
\section{Crossed ideal maps between ideal maps}
 As it is well known and explored above that  an ideal structure over a homomorphism of algebras $\eta:R\rightarrow S$ preserves the ideals. So we can say that  if there is an ideal structure  over $\eta$, then $\eta(R)$ is an ideal of $S$. This ideal approach to crossed modules shades some light on Loday's crossed square (cf. \cite{walery}). In this section we will give an extension of this result for higher dimensional crossed modules. We will prove that if there is a (crossed) ideal structure over a morphism between crossed modules, then this map preserves the (crossed) ideals.  First we give the notion of  `crossed ideal' of a crossed module of algebras from \cite{nizar}.
\begin{definition}
A homomorphism of algebras $\eta':R'\rightarrow S'$  will be called a \textit{crossed ideal} of the crossed module $\eta:R\rightarrow S$ in the category of crossed modules over $k$-algebras if:

$\mathfrak{CI}1:$ $\eta':R'\rightarrow S'$ is a subcrossed module of $\eta:R\rightarrow S$, that is, the following conditions are satisfied:

\ \  \ $(i)$ $R'$ is a subalgebra of  $R$ and  $S'$ is a subalgebra of $S$.

\ \ \ $(ii)$ the action of $S'$ on $R'$ induced by the action of $S$ on $R$.

\ \ \ $(iii)$ $\eta':R'\rightarrow S'$ is a crossed module.

\ \ \ $(iv)$ the following diagram of morphisms of crossed modules
\begin{equation*}
\xymatrix{ R' \ar[d]_{\eta'} \ar[r]^{\mu} & R \ar[d]^{\eta}
\\ S' \ar[r]_{\nu} & S}
\end{equation*}%
commutes, where $\mu$ and $\nu$ are the inclusions,

$\mathfrak{CI}2:$ $R'R=RR'\subseteq R'$ and $SS'=S'S\subseteq S'$,

$\mathfrak{CI}3:$ $R\cdot S'=S'\cdot R \subseteq R'$,

$\mathfrak{CI}4:$ $R'$ is closed under the action of $S$, i.e. $S\cdot R'=R'\cdot S \subseteq R'$.
\end{definition}

\subsection{Crossed ideal structure over maps between crossed modules}

 Assume that $\eta_1:R_1\rightarrow S_1$ and $\eta_2:R_2\rightarrow S_2$ are crossed modules. Let $\alpha:(\alpha_1,\alpha_2)$ be a morphism from $\eta_1$ to $\eta_2$ in the category of crossed modules of $k$-algebras, where $\alpha_1:R_1\rightarrow R_2$ and $\alpha_2:S_1\rightarrow S_2$ are  homomorphisms of $k$-algebras. In this case, the morphism $\alpha:=(\alpha_1,\alpha_2)$ satisfies the following conditions:

$(i)$ the  diagram
$$
\xymatrix{ R_1 \ar[d]_{\eta_1} \ar[r]^{\alpha_1} & R_2 \ar[d]^{\eta_2}
\\ S_1 \ar[r]_{\alpha_2} & S_2}
$$
commutes, i.e. $\alpha_2\eta_1=\eta_2\alpha_1$,

$(ii)$ for all $s_1\in S_1$ and $r_1\in R_1$,
$$
\alpha_1(l_{s_1}(r_1))=l_{\alpha_2(s_1)}(\alpha_1(r_1)).
$$

\begin{definition} A morphism $\alpha:=(\alpha_1,\alpha_2)$ between crossed modules $\eta_1$ and $\eta_2$ is called a `crossed ideal map' if

$(i)$ there are ideal map structures over the homomorphisms $\alpha_1,\alpha_2$ and $\eta_2\alpha_1=\alpha_2\eta_1$, and

$(ii)$ there is an $S_2$-bilinear map $h:R_2\times S_1\rightarrow R_1$ satisfying the conditions:

\ \ \ $(a)$ $\alpha_1(h(r_2,s_1))=l_{\alpha_2(s_1)}(r_2)$,

\ \ \ $(b)$ $  \eta_1(h(r_2,s_1))=l_{\eta_2(r_2)}(s_1)$,

\ \ \ $(c)$ $h(\alpha_1(r_1),s_1)=l_{s_1}(r_1)$,

\ \ \ $(d)$ $h(r_2,\eta_1(r_1))=l_{r_2}(r_1).$\\
for all $r_2\in R_2,s_1\in S_1$.
\end{definition}
\begin{remark}\rm{}
In fact,  a crossed ideal structure over the map $\alpha$, between crossed modules $\eta_1$ and $\eta_2$, gives  a crossed square structure  of algebras on the
square
$$
\xymatrix{ R_1 \ar[d]_{\eta_1} \ar[r]^{\alpha_1} & R_2 \ar[d]^{\eta_2}
\\ S_1 \ar[r]_{\alpha_2} & S_2}
$$
defined by Ellis \cite{ellis} and  introduced by Guin-Wal\'{e}ry and Loday, \cite{walery}, in the group case.
\end{remark}
Thus we get the following result.

\begin{proposition}\rm{}
If the morphism $\alpha:(\alpha_1,\alpha_2)$ is a crossed  ideal map from  $\eta_1:R_1\rightarrow S_1$ to  $\eta_2:R_2\rightarrow S_2$ in the category of crossed modules of $k$-algebras, then $\alpha(\eta_1):\alpha_1(R_1)\rightarrow \alpha_2(S_1)$ is a crossed ideal of the crossed module $\eta_2:R_2\rightarrow S_2$.
\end{proposition}
\begin{proof}
First, we consider the following  square
$$
\xymatrix{ \alpha_1(R_1)=R'_1 \ar[d]_{\overline{\eta_2}} \ar[r]^-{\mu} & R_2 \ar[d]^{\eta_2}
\\ \alpha_2(S_1)=S'_1 \ar[r]_-{\nu} & S_2}
$$
where $\mu$ and $\nu$ are the inclusions. The map $\overline{\eta_2}:R'_1\rightarrow S'_1$ is defined by the restriction of the map $\eta_2$ to the subalgebra $\alpha_1(R_1)$ of $R_2$. We will show that the restricted homomorphism $\overline{\eta_2}$ is a crossed ideal of $\eta_2$.

$\mathfrak{CI}1.$ First we will show that $\overline{\eta_2}$ is a subcrossed module of $\eta_2$.

\ \ \ $(i)$ It is clear that $R'_1$ is a subalgebra of $R_2$ and similarly $\alpha_2(S_1)=S_1'$ is a subalgebra of $S_2$.

\  \  \ $(ii)$ Since the map $\alpha:=(\alpha_1,\alpha_2)$ is a crossed module morphism, it satisfies the condition
$l_{\alpha_2(s_1)}(\alpha_1 r_1)=\alpha_1(l_{s_1}(r_1))$ for all $r_1\in R_1$ and $s_1\in S_1$. Then the algebra action of $\alpha_2(s_1)\in S'_1$ on
$\alpha_1(r_1)\in R'_1$ can be given by $\alpha_2(s_1)\cdot \alpha_1(r_1)=\alpha_1(s_1\cdot r_1)\in R'_1$.

\ \ \ $(iii)$ We will show that $\overline{\eta_2}:R'_1\rightarrow S'_1$ is a crossed module. For all $\alpha_2(s_1)\in S'_1$ and $\alpha_1(r_1), \alpha_1(r'_1)\in R'_1$,  we obtain
\begin{align*}
\overline{\eta_2}(l_{\alpha_2(s_1)}(\alpha_1(r_1)))&=\eta_2\alpha_1(l_{s_1}(r_1))\\
&=\alpha_2\eta_1(l_{s_1}(r_1))\\
&=\alpha_2(s_1\eta_1(r_1)) \ \ (\text{ since } \eta_1 \text{ is a crossed module })\\
&=\alpha_2(s_1)\alpha_2\eta_1(r_1)\\
&=\alpha_2(s_1)\eta_2\alpha_1(r_1)\\
&=\alpha_2(s_1)\overline{\eta_2}(\alpha_1(r_1)),
\intertext{and}
l_{\overline{\eta_2}(\alpha_1(r_1))}\alpha_1(r'_1)&=l_{\alpha_2(\eta_1(r_1))}\alpha_1(r'_1)\\
&=\alpha_1(l_{\eta_1(r_1)}(r'_1))\\
&=\alpha_1(r_1r'_1)\ \ (\text{ since } \eta_1 \text{ is a crossed module })\\
&=\alpha_1(r_1)\alpha_1(r'_1).
\end{align*}

\ \ \ $(iv)$ the square
$$
\xymatrix{ R'_1 \ar[d]_{\overline{\eta_2}} \ar[r]^-{\mu} & R_2 \ar[d]^{\eta_2}
\\ S'_1\ar[r]_-{\nu} & S_2}
$$
is commutative, because $\mu $ and $\nu$ are the inclusions and $\overline{\eta_2}$ is given by the restriction of $\eta_2$. Thus $\overline{\eta_2}$ is a subcrossed module of $\eta_2$.

$\mathfrak{CI}2.$ Since there are ideal structures over the maps $\alpha_1:R_1\rightarrow R_2$ and $\alpha_2:S_1\rightarrow S_2$, we obtain that
$\alpha_1(R_1)=R'_1$ and $\alpha_2(S_1)=S'_1$ are ideals of $R_2$ and $S_2$ respectively. Therefore, we obtain
$$
R'_1R_2=R_2R'_1\subseteq R'_1 \text{ and } S'_1S_2=S_2S'_1\subseteq S'_1.
$$

$\mathfrak{CI}3.$ We have to show that $R_2\cdot S'_1=S'_1\cdot R_2\subseteq R'_1$. We use the $h$-map  to prove it. For all $\alpha_2(s_1)\in S'_1$ and $r_2\in R_2$ we have $r_2\cdot \alpha_2(s_1)=\alpha_2(s_1)\cdot r_2=l_{\alpha_2(s_1)}(r_2)=\alpha_1(h(r_2,s_1))$, where $h(r_2,s_1)\in R_1$, then we obtain  $r_2\cdot \alpha_2(s_1)=\alpha_2(s_1)\cdot r_2\in \alpha_1(R_1)=R'_1$ so that $R_2\cdot S'_1=S'_1\cdot R_2\subseteq R'_1$.

$\mathfrak{CI}4.$ We have to show that $S_2\cdot R'_1=R'_1\cdot S_2\subseteq R'_1$. For all $s_2\in S_2$ and $\alpha_1(r_1)\in R'_1$, we can define the action
by $s_2\cdot \alpha_1(r_1)=l_{s_2}(\alpha_1(r_1))=\alpha_1(l_{s_1}(r_1))\in R'_1$. Thus $R'_1$ is closed under the action of $S_2$ and this completes the proof.
\end{proof}

Conversely, as it was illustrated in \cite{nizar}, we can easily say that given  a crossed ideal $\overline{\eta_2}:R'_1\rightarrow S'_1$ of the crossed module $\eta_2:R_2\rightarrow S_2$, then inclusion morphism from $\overline{\eta_2}$ to $\eta_2$ is  a crossed ideal map in the category of crossed modules of $k$-algebras.

Indeed, in the following diagram
$$
\xymatrix{ R'_1 \ar[d]_{\overline{\eta_2}} \ar[r]^-{\mu} & R_2 \ar[d]^{\eta_2}
\\ S'_1 \ar[r]_-{\nu} & S_2}
$$
if $\overline{\eta_2}$ is a crossed ideal of $\eta_2$, the inclusion morphisms $\mu$ and $\nu$ become crossed modules with the natural actions of $R_2$ and $S_2$ on their ideals $R'_1$ and $S'_1$ given by the multiplication, respectively. Further, the $h$-map $h:R_2\times S'_1\rightarrow R'_1$ is defined by $h(r_2,s'_1)=(l|_{S'_1})_{s'_1}(r_2)$, where $l|_{S'_1}$ is the restriction of $l:S_2\rightarrow End(R_2)$ to $S'_1$.

\section{From the morphism $\alpha:\eta_1\rightarrow \eta_2$ to the usual Bar construction}

As mentioned above, in \cite{far1}, Farjoun and Segev proved that a crossed module map $l:G\rightarrow Aut(N)$, which they call a normal structure on the map $N\rightarrow G$ is inversely associated with a group structure on the homotopy quotient $G//N:=hocomlim_{N}G$ by taking  $G//N$ to be the usual Bar construction.  They also stated in section 6 of their work, for a morphism from a normal map $X\rightarrow G$ to a normal map $Y\rightarrow H$ in the category of normal maps, one can form a simplicial group morphism
$X//G \rightarrow Y//H$ and the homotopy quotient $(Y//H)//(X//G)$.  In fact, if there is a normal map structure over the simplicial group morphism $X//G \rightarrow Y//H$, then $(Y//H)//(X//G)$ becomes a \textit{bisimplicial} group, algebraically. In this section, we  make some remarks concerning these ideas over $k$-algebras.

Recall that a  functor $\mathbf{E.,.}:(\Delta \times \Delta)^{op}\rightarrow \mathbf{Alg}$ is called a bisimplicial algebra, where $\Delta$ is the category of
finite ordinals and $\mathbf{Alg}$ is the category of (commutative) $k$-algebras.
 Hence $\mathbf{E.,.}$ is equivalent to giving for each $(p,q)$ an algebra  $E_{p,q}$ and morphisms:
\begin{align*}
d_{i}^{h^{(pq)}}:E_{p,q}\rightarrow E_{p-1,q} &  \\
s_{i}^{h^{(pq)}}:E_{p,q}\rightarrow E_{p+1,q} & \quad 0\leq i\leq p \\
d_{j}^{v^{(pq)}}:E_{p,q}\rightarrow E_{p,q-1} &  \\
s_{j}^{v^{(pq)}}:E_{p,q}\rightarrow E_{p,q+1} & \quad 0\leq j\leq q
\end{align*}
such that the maps $d_{i}^{h^{(pq)}},s_{i}^{h^{(pq)}}$ commute with $d_{j}^{v^{(pq)}},s_{j}^{v^{(pq)}}$
and that $d_{i}^{h^{(pq)}},s_{i}^{h^{(pq)}}$ (resp. $d_{j}^{v^{(pq)}},s_{j}^{v^{(pq)}}$) satisfy the
usual simplicial identities.

 Now we suppose that  $\alpha:(\alpha_1,\alpha_2)$ is a morphism from  $\eta_1:R_1\rightarrow S_1$ to  $\eta_2:R_2\rightarrow S_2$ in the category of crossed modules of $k$-algebras. Using the usual Bar construction, we can form the simplicial algebras $S_1//R_1$ and $S_2//R_2$ from  $\eta_1$ and $\eta_2$ respectively as above. Analogously to \cite{far1}, we obtain a simplicial algebra morphism
$$
\Phi:S_1//R_1 \rightarrow S_2//R_2
$$
and we can define this map on each step by
$$
\Phi_n:(S_1\ltimes (R_1)^{\ltimes^n}) \rightarrow (S_2\ltimes (R_2)^{\ltimes^n} )
$$
with
 $$\Phi_n:(s_1,r_1,r_2,\dots,r_n)=(\alpha_2(s_1),\alpha_1(r_1),\alpha_1(r_2),\dots,\alpha_1(r_n))$$
for all $s_1\in S_1$ and $r_i\in R_1$ and where the maps $\Phi_n$ are homomorphisms of algebras.

An action of the algebra $(S_1\ltimes (R_1)^{\ltimes^n})$ on
the underlying $k$-module of the algebra $(S_2\ltimes (R_2)^{\ltimes^n})$ can be given by this map, namely,
$$
(s_1,\ltimes_{i=1}^{n}(r_i)):(s_2,\ltimes_{i=1}^{n}(r'_i))=(s_2+\alpha_1(s_1),\ltimes_{i=1}^{n}(r'_i+\alpha_1(r_i)))
$$
where $s_1\in S_1, s_2\in S_1$ and $r_i\in R_1$, $r'_i\in R_2$ for $i=1,2,\dots,n$.

Using this action, on each step and considering the usual Bar construction, we can form a \textit{bisimplicial} $k$-module,
$$
\mathfrak{Bar}^{(2)}:(S_2//R_2)//(S_1//R_1)
$$
and, on each directions, this can be defined by the $k$-modules
$$
\mathfrak{Bar}^{(2)}_{n,m}:=(S_2\ltimes (R_2)^{\ltimes^n})\times (S_1\ltimes (R_1)^{\ltimes^n})^{\times^m}.
$$

The horizontal  homomorphisms between these $k$-modules can be defined as follows:

1. For all
$$
(s_2,r_{21},\cdots,r_{2n})\in S_2\ltimes (R_2)^{\ltimes^n}
$$
and
$$((s_1^1,r_{11}^1,\cdots,r_{1n}^1),\cdots,(s_1^m,r_{11}^m,\cdots,r_{1n}^m) )\in (S_1\ltimes (R_1)^{\ltimes^n})^{\times^m},$$
where, for $1\leqslant i \leqslant n$ and $1\leqslant j \leqslant m$, $r_{1i}^j\in R_1$,  $r_{2i}\in R_2$, $s_2\in S_2$, $s_1^j\in S_1$, the $d_0^h:\mathfrak{Bar}^{(2)}_{n,m}\rightarrow \mathfrak{Bar}^{(2)}_{n,m-1}$ is defined by
\begin{multline*}
d_0^h((s_2,r_{21},\cdots,r_{2n}),(s_1^1,r_{11}^1,\cdots,r_{1n}^1),\cdots,(s_1^m,r_{11}^m,\cdots,r_{1n}^m) )\\
\begin{aligned}
&=\left((s_2,r_{21},\cdots,r_{2n})+\Phi_n(s_1^1,r_{11}^1,\cdots,r_{1n}^1), (s_1^2,r_{11}^2,\cdots,r_{1n}^2),\cdots,(s_1^m,r_{11}^m,\cdots,r_{1n}^m)\right).
\end{aligned}
\end{multline*}

2. For $0<i<m$, the $d_i^h:\mathfrak{Bar}^{(2)}_{n,m}\rightarrow \mathfrak{Bar}^{(2)}_{n,m-1}$ is defined by
\begin{multline*}
d_i^h((s_2,r_{21},\cdots,r_{2n}),(s_1^1,r_{11}^1,\cdots,r_{1n}^1),\cdots,(s_1^m,r_{11}^m,\cdots,r_{1n}^m) )\\
\begin{aligned}
&=((s_2,r_{21},\cdots,r_{2n}),(s_1^1,r_{11}^1,\cdots,r_{1n}^1),\cdots,  \\
& \ \ (s_1^i,r_{11}^i,\cdots,r_{1n}^i)+(s_1^{i+1},r_{11}^{i+1},\cdots,r_{1n}^{i+1}),\cdots,(s_1^m,r_{11}^m,\cdots,r_{1n}^m)).
\end{aligned}
\end{multline*}

3. $d_m^h:\mathfrak{Bar}^{(2)}_{n,m}\rightarrow \mathfrak{Bar}^{(2)}_{n,m-1}$ is defined by
\begin{multline*}
d_m^h((s_2,r_{21},\cdots,r_{2n}),(s_1^1,r_{11}^1,\cdots,r_{1n}^1),\cdots,(s_1^m,r_{11}^m,\cdots,r_{1n}^m) )\\
\begin{aligned}
&=((s_2,r_{21},\cdots,r_{2n}),(s_1^1,r_{11}^1,\cdots,r_{1n}^1),\cdots,(s_1^{m-1},r_{11}^{m-1},\cdots,r_{1n}^{m-1})).
\end{aligned}
\end{multline*}

4. For all $0\leqslant i \leqslant m$, the $s_i^h:\mathfrak{Bar}^{(2)}_{n,m}\rightarrow \mathfrak{Bar}^{(2)}_{n,m+1}$ is defined by
\begin{multline*}
s_i^h((s_2,r_{21},\cdots,r_{2n}),(s_1^1,r_{11}^1,\cdots,r_{1n}^1),\cdots,(s_1^m,r_{11}^m,\cdots,r_{1n}^m) )\\
\begin{aligned}
&=((s_2,r_{21},\cdots,r_{2n}),(s_1^1,r_{11}^1,\cdots,r_{1n}^1),\cdots,\\
&\ \ (s_1^{i},r_{11}^{i},\cdots,r_{1n}^{i}),(0,0,\cdots,0),
(s_1^{i+1},r_{11}^{i+1},\cdots,r_{1n}^{i+1}),\cdots, (s_1^m,r_{11}^m,\cdots,r_{1n}^m) ).
\end{aligned}
\end{multline*}

Similarly, the vertical homomorphisms can be defined as follows:

1. the $d_0^v:\mathfrak{Bar}^{(2)}_{n,m}\rightarrow \mathfrak{Bar}^{(2)}_{n-1,m}$ is defined by
\begin{multline*}
d_0^v((s_2,r_{21},\cdots,r_{2n}),(s_1^1,r_{11}^1,\cdots,r_{1n}^1),\cdots,(s_1^m,r_{11}^m,\cdots,r_{1n}^m) )\\
\begin{aligned}
&=\left((s_2+\eta_2(r_{21}),r_{22}\cdots,r_{2n}), (s_1^1+\eta_1(r_{11}^1),r_{12}^2\cdots,r_{1n}^2),\cdots,(s_1^m+\eta_1(r_{11}^m),r_{12}^m\cdots,r_{1n}^m)\right).
\end{aligned}
\end{multline*}

2. For $0<i<n$, the $d_i^v:\mathfrak{Bar}^{(2)}_{n,m}\rightarrow \mathfrak{Bar}^{(2)}_{n-1,m}$ is defined by
\begin{multline*}
d_i^v((s_2,r_{21},\cdots,r_{2n}),(s_1^1,r_{11}^1,\cdots,r_{1n}^1),\cdots,(s_1^m,r_{11}^m,\cdots,r_{1n}^m) )\\
\begin{aligned}
&=((s_2,r_{21},\cdots,r_{2i}+r_{2(i+1)},\cdots, r_{2n}),(s_1^1,r_{11}^1,\cdots,r_{1i}^1+r_{1(i+1)}^1,\cdots,r_{1n}^1),\cdots,  \\
& \ \ (s_1^m,r_{11}^m,\cdots,r_{1i}^m+r_{1(i+1)}^m\cdots,r_{1n}^m)).
\end{aligned}
\end{multline*}

3.  $d_n^v:\mathfrak{Bar}^{(2)}_{n,m}\rightarrow \mathfrak{Bar}^{(2)}_{n-1,m}$ is defined by
\begin{multline*}
d_n^v((s_2,r_{21},\cdots,r_{2n}),(s_1^1,r_{11}^1,\cdots,r_{1n}^1),\cdots,(s_1^m,r_{11}^m,\cdots,r_{1n}^m) )\\
\begin{aligned}
&=((s_2,r_{21},\cdots,r_{(2n-1)}),(s_1^1,r_{11}^1,\cdots,r_{1(n-1)}^1),\cdots,(s_1^{m-1},r_{11}^{m},\cdots,r_{1(n-1)}^{m})).
\end{aligned}
\end{multline*}

4.For all $0\leqslant i \leqslant n$, the $s_i^v:\mathfrak{Bar}^{(2)}_{n,m}\rightarrow \mathfrak{Bar}^{(2)}_{n+1,m}$ is defined by
\begin{multline*}
s_i^v((s_2,r_{21},\cdots,r_{2n}),(s_1^1,r_{11}^1,\cdots,r_{1n}^1),\cdots,(s_1^m,r_{11}^m,\cdots,r_{1n}^m) )\\
\begin{aligned}
&=((s_2,r_{21},\cdots,r_{2i},0,r_{2(i+1)},\cdots,r_{2n}),(s_1^1,r_{11}^1,\cdots,r_{1i}^1,0,r_{1(i+1)}^1,\cdots,r_{1n}^1),\cdots,\\
&\ \  (s_1^m,r_{11}^m,\cdots,r_{1i}^m,0,r_{1(i+1)}^m,\cdots r_{1n}^m) ).
\end{aligned}
\end{multline*}

In low dimensions, we can illustrate this bisimplicial $k$-module by the diagram:
\begin{equation*}
\xymatrix{ \ldots \ar@<-1.5ex>[d]\ar@<-0.5ex>[d] \ar@<0.5ex>[d] \ar@
<1ex>[r] \ar@<0ex>[r] \ar@<-1ex>[r] &
(S_2\ltimes (R_2)^{2})\times(S_1\ltimes(R_1)^{2}) \ar@<-1.5ex>[d]\ar@<-0.5ex>[d]
\ar@<0.5ex>[d] \ar@ <0.5ex>[r] \ar@<-0.5ex>[r] & (S_2\ltimes R_2^{2})\ar@<-1.5ex>[d]\ar@<-0.5ex>[d] \ar@<0.5ex>[d]
\\ (S_2\ltimes R_2)\times (S_1\ltimes R_1)^{2} \ar@<-0.5ex>[d]
\ar@<0.5ex>[d] \ar@ <1ex>[r] \ar@<0ex>[r] \ar@<-1ex>[r] &
(S_2\ltimes R_2)\times(S_1\ltimes R_1) \ar@<-0.5ex>[d] \ar@<0.5ex>[d]\ar@ <0.5ex>[r]
\ar@<-0.5ex>[r] & S_2\ltimes R_2\ar@<-0.5ex>[d] \ar@<0.5ex>[d]
\\ \ldots S_2\times (S_1)^{2} \ar@
<1ex>[r] \ar@<0ex>[r] \ar@<-1ex>[r] & (S_2\times S_1) \ar@ <1ex>[r]
\ar@<0ex>[r] & S_2}
\end{equation*}%
For instance, in this diagram, the homomorphisms in the first square are given by:
$$
  \begin{array}{rl}
    d_0^v(s_2,r_2)=s_2+\eta_2r_2, &d_0^h(s_2,s_1)=s_2+\alpha_2(s_1)  \\
    d_1^v(s_2,r_2)=s_2, & d_1^h(s_2,s_1)=s_2 \\
    s_0^v(s_2)=(s_2,0), & s_0^h(s_2)=(s_2,0). \\
  \end{array}
$$
and
$$
  \begin{array}{rl}
    d_0^v(s_2,r_2,s_1,r_1)=(s_2+\eta_2r_2,s_1+\eta_1r_1), &d_0^h(s_2,r_2,s_1,r_1)=(s_2+\alpha_2(s_1),r_2+\alpha_1(r_1))  \\
    d_1^v(s_2,r_2,s_1,r_1)=(s_2,s_1), & d_1^h(s_2,r_2,s_1,r_1)=(s_2,r_2) \\
    s_0^v(s_2,s_1)=(s_2,0,s_1,0), & s_0^h(s_2,r_2)=(s_2,r_2,0,0). \\
  \end{array}
$$

Therefore, we obtained a bisimplicial $k$-module, from  the morphism $\alpha$ in the
category of crossed modules of $k$-algebras. Thus we get the following result.

\begin{theorem}\rm{}
Given a morphism $\alpha:\eta_1\rightarrow \eta_2$ in the category of crossed modules of $k$-algebras, \textit{a crossed ideal map structure} on the morphism $\alpha$ gives  \textit{an ideal bisimplicial algebra structure}  on the bisimplicial $k$-module $\mathfrak{Bar}^{(2)}:(S_2//R_2)//(S_1//R_1)$, and conversely, \textit{any ideal bisimplicial algebra structure} on the bisimplicial $k$-module $\mathfrak{Bar}^{(2)}:(S_2//R_2)//(S_1//R_1)$ determines \textit{a crossed ideal map  structure} on the morphism $\alpha:\eta_1\rightarrow \eta_2$.
\end{theorem}

\begin{remark}\rm{}
Obviously, proving this theorem would take too much  work and time.  In order to prove it, we would need to give the notion of  `\textit{ideal bisimplicial algebra structure}' over the bisimplicial $k$-module  $\mathfrak{Bar}^{(2)}$ explicitly.    So we will clarify it in another work. Of course, this result can be iterated to the crossed $n$-cube structure defined by Ellis in \cite{ellis}. In this case, we would need to give a detailed  definition of  a \textit{crossed $n$-ideal} of a crossed $n$-cube and a \textit{crossed $n$-ideal structure} over the morphism between crossed $(n-1)$ cubes.  Then it would give a multi-simplicial algebra in dimension $n$, or an $n$-simplicial algebra together with this structure.
\end{remark}

$
\begin{array}{lll}
\text{Alper Odaba{\c s} }   & & \text{Erdal Ulualan}\\
\text{Eski\c{s}ehir Osmangazi University},& &  \text{Dumlup\i nar University} \\
\text{Science and Art Faculty}  & & \text{Science and Art Faculty}\\
\text{Department of Mathematics-Computer} & & \text{Mathematics Department}\\
\text{26480, Eski\c{s}ehir, TURKEY}  & & \text{K\"{u}tahya, TURKEY} \\
\text{e-Mail: aodabas@ogu.edu.tr,} & &
\text{e-Mail: erdal.ulualan@dpu.edu.tr}\\
\end{array}%
$

\end{document}